\documentclass[11pt]{amsart}
\usepackage{amssymb,amsthm,amsmath,amstext}
\usepackage{mathrsfs}  
\usepackage{bm}        
\usepackage{mathtools} 
\usepackage{color}
\usepackage[bbgreekl]{mathbbol}
\usepackage{tabularx}

\usepackage{cite}

\usepackage[left=1.5in,top=1in,right=1.5in,bottom=1in]{geometry}

\usepackage{graphicx}

\usepackage[all]{xy}
\usepackage{tikz}

\newcommand{\xycenter}[1]{
	\begin{center}
	\mbox{\xymatrix{#1}}
	\end{center}
	}

\theoremstyle{plain}
\newtheorem{theorem}{Theorem}[section]

\newtheorem{proposition}[theorem]{Proposition}
\newtheorem{lemma}[theorem]{Lemma}
\newtheorem{corollary}[theorem]{Corollary}

\theoremstyle{definition}
\newtheorem{definition}[theorem]{Definition}

\newtheorem{example}[theorem]{Example}

\theoremstyle{remark}

\newcommand{\C}{\mathbb C}
\renewcommand{\P}{\mathbb P}

\newcommand{\Hom}{\mathrm{Hom}}
\newcommand{\id}{\mathrm{id}}

\newcommand{\Pf}{\mathrm{Pf}}

\usepackage[backref=page]{hyperref}

\newcommand{\GL}{\mathrm{GL}}

\begin{document} 

\title[Formats of $6\times 6$ skew matrices ]{Formats of $6\times 6$ skew matrices of linear forms with vanishing Pfaffian}

\author[B\"ohning]{Christian B\"ohning}\thanks{The first author was supported by the EPSRC New Horizons Grant EP/V047299/1.}
\address{Christian B\"ohning, Mathematics Institute, University of Warwick\\
Coventry CV4 7AL, England}
\email{C.Boehning@warwick.ac.uk}

\author[von Bothmer]{Hans-Christian Graf von Bothmer}
\address{Hans-Christian Graf von Bothmer, Fachbereich Mathematik der Universit\"at Hamburg\\
Bundesstra\ss e 55\\
20146 Hamburg, Germany}
\email{hans.christian.v.bothmer@uni-hamburg.de}

\date{\today}


\begin{abstract}
We show that every skew-symmetric $6 \times 6$ matrix of linear forms with vanishing Pfaffian is congruent to one of finitely many types of matrices, each of which is characterised by a specific pattern of zeroes (and some other linear relations) among its entries. Such matrices are for example important for compactifying moduli spaces of stable rank $2$ vector bundles with Chern classes $c_1=0, c_2=2$  on  cubic threefolds. 
\end{abstract}

\maketitle

In this paper we prove that a skew-symmetric $6 \times 6$ matrix of linear forms with vanishing Pfaffian has at least one of the  skew formats in Table \ref{tFormats}. This result is known by work of Manivel and Mezzetti \cite{MaMe05} for matrices of linear forms whose rank at most $2$ locus is empty. This condition implies, in  particular, that the entries of such a matrix span at most a $\P^2$. Wo do not make any additional assumptions in this paper. 

Our result is also known by work of Comaschi \cite{Co21} if one assumes that the matrix is stable. This condition implies that the span of the entries of such matrices is at least a $\P^3$ and that formats $(a), (b)$ and $(c)$ do not occur. We extend Comaschi's classification to the strictly semistable case in which only format $(e)$ is possible. 

In fact Manivel, Mezzetti and Comaschi prove more: they show that the matrices of each format form an irreducible orbit under an appropriate group and even give a normal form in each case. While it is clear that the matrices of each format form an orbit of a projective space under the $\GL_6(\C)$ operation $A \mapsto SAS^t$ we do not attempt to give normal forms in the general case. We do find normal forms in the strictly semistable case, and also obtain a slight variation of Comaschi's normal forms for stable $M$ from our methods. 

Notice that it is obvious that matrices of the formats given in Table \ref{tFormats} have vanishing Pfaffians. The difficulty lies in proving that there are no others.

\begin{table}[h!]
\begin{tabular}{|c|c|}
\hline
(a) &
$\left(\begin{smallmatrix}
	0 & * & * & * & * &  \\
	* & 0 &* & * & * &  \\
	* & * & 0 &* & * &  \\
	* & * & * & 0& * &  \\
	* & * & * & * & 0 & \\
	\\
\end{smallmatrix}\right)$
\\ \hline
(b) &
$\left(\begin{smallmatrix}
	& * & * & * & * & * \\
	*&&* & * & * & * \\
	* & * & \\
	* & * & \\
	* & * & \\
	* & * & 
\end{smallmatrix}\right)$
\\  \hline
(c) &
$\left(\begin{smallmatrix}
	0 & * & * & * & * & * \\
	* & 0 &* & * &  \\
	* & * & 0 &* &  \\
	* & * & * &0 & \\
	* \\
	*
\end{smallmatrix}\right)$
\\  \hline
\end{tabular}
\quad
\begin{tabular}{|c|c|c|}
\hline
(d) &
$\left(\begin{smallmatrix}
	0 & * & * & \\
	* & 0 &* &\\
	* & * & 0 &  \\
	&&&0 & * & * \\
	&& & * & 0 &* \\
	&& & * & * & 0
\end{smallmatrix}\right)$
\\  \hline
(e) &
$\left(\begin{smallmatrix}
	0 & * & * & 0 & * & *\\
	* & 0 &* & *& 0 & * \\
	* & * & 0 & * & * & 0 \\
	0 & * & * \\
	 * & 0 &* \\
	 * & * & 0 
\end{smallmatrix}\right)$
\\  \hline
(f) &
$\left(\begin{smallmatrix}
	& * &  & 0 & * & *\\
	* &  & & *& 0 & * \\
	 &  &  & * & * & 0 \\
	0 & * & * &  & * & \\
	 * & 0 &* & * &  & \\
	 * & * & 0 &  &  & 
\end{smallmatrix}\right)$
\\\hline
\end{tabular}
\caption{Formats of skew-symmetric $6 \times 6$ matrices of linear forms with vanishing Pfaffian. The last three
are to be understood as double skew formats.}
\label{tFormats}
\end{table}

\section{Formats}

Let $V$ be a finite-dimensional $\C$-vector space, $\P(V)$ the associated projective space. We say that $M$ is a matrix of linear forms over $\P(V)$ if it has entries in $V^*$. 
For any matrix $M$ we denote the entry in row $i$ and column $j$ by  $m_{ij}$. We think of the spaces of rows and columns as projective spaces and therefore let row and column indices start with zero. Our notation for the span of these entries is 
$\langle M \rangle := \langle m_{ij} \rangle\subset V^*$. Notice that this span could be a proper subspace of $V^*$.

For a skew-symmetric matrix $M$ we denote by $M_{ij}$ the matrix obtained by deleting the $i$th and $j$th row and column of $M$. We also denote by $\Pf(M)$ the Pfaffian of $M$ if $M$ is skew. 

We now introduce the language of {\sl formats} of matrices.

\begin{definition}[Formats of matrices]\label{dFormat}
 Let $M$ be an $r \times s$ matrix of linear forms over $\P(V)$ and $F$ an $r \times s$ matrix with entries $0$ or $*$. 
\begin{enumerate}
\item We say that $M$ is of the {\sl form}, or  has {\sl form}, $F$ if $m_{ij}$ is $0$ whenever $f_{ij}$ is $0$. In particular, $M$ may also have $0$ entries where $F$ has $*$ entries. 
\item  We say that $M$ has {\sl format} $F$ if there exist matrices $S\in \GL_{r}(\C)$ and $T\in \GL_{s} (\C )$ such that $SM T$ hasr form $F$.
\item If $r=s$ and $M$ is skew, we say that $M$ has {\sl skew format} $F$ if there exists $S \in \GL_r (\C )$ such that $S M S^t$ has form $F$.
\end{enumerate}
Sometimes we replace a $0$ by a blank space if we feel that this improves the readability.
\end{definition}

\begin{example}
$M = \left(\begin{smallmatrix} x & 0 \\ y & 0 \end{smallmatrix}\right)$ has {\sl form} $ \left(\begin{smallmatrix} * & 0 \\ * & 0 \end{smallmatrix}\right)$. $M' = \left(\begin{smallmatrix} x & x\\ y & y \end{smallmatrix}\right)$ has {\sl format} $ \left(\begin{smallmatrix} * & 0 \\ * & 0 \end{smallmatrix}\right)$, since it can be transformed into $M$ by row and column operations. 
\end{example}

\begin{definition}
In this paper we consider $6 \times 6$ skew matrices $M$ of linear forms over a projective space $\P(V)$. In this special situation we say that $M$ is {\sl double skew} if there are skew symmetric $3 \times 3$ matrices of linear forms $N_0,N_1,N_2$ over $\P(V)$ such that
\[
	M = \begin{pmatrix} N_0 & N_1 \\ -N_1^t & N_2 \end{pmatrix} =
	  \begin{pmatrix} N_0 & N_1 \\ N_1 & N_2 \end{pmatrix}
\]
We say that $M$ is of {\sl double skew format}  $F$ if there exists a matrix $S \in \GL_6 (\C )$ such that $S M S^t$ is double skew and has form $F$.
\end{definition}

For future reference we recall the analogue of Laplace expansion for Pfaffians. 

\begin{proposition}\label{pPfaffianLaplace}
Let $M$ be a skew-symmetric $2n\times 2n$-matrix (with entries in some commutative ring), and $i\in \{0, \dots 2n-1\}$ fixed.
Then
\[
\Pf (M)=\sum_{j\neq i}(-1)^{i+j+1+\theta(i, j)}m_{ij}\Pf(M_{ij})
\]
where $\theta (i, j)= 0$ if $i\le j$ and $1$ otherwise. 
\end{proposition}

\begin{proof}
See \cite[equation (D.1) p. 116]{FP98}.
\end{proof}

\section{proof of the Theorem}

\begin{figure}
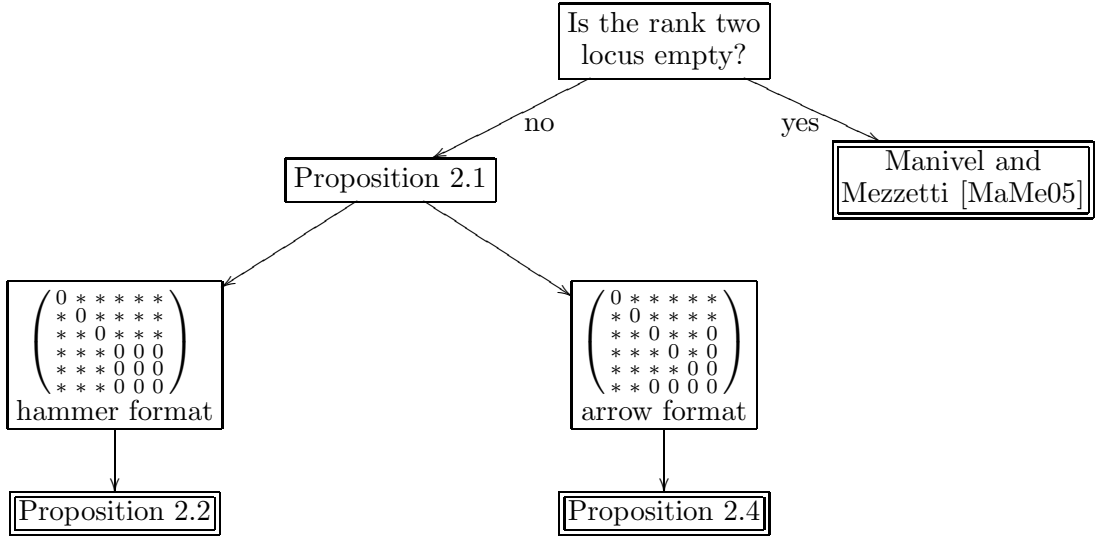

\xycenter{
	&&*+[F]\txt{Is the rank two\\ locus empty?} \ar[dl]^{\txt{no}}  \ar[dr]_{\txt{yes}} 
	\\
	&*+[F]\txt{Proposition \ref{pRank2}}
	\ar[dl]
	\ar[dr]
	&&*+[F=]\txt{Manivel and \\  Mezzetti \cite{MaMe05}}
	 \\
	*+[F]\txt{\text{$
	\left(
	\begin{smallmatrix} 
	0 & * & * & * & * & *\\
	* & 0 & * & * & * & * \\
	* & * & 0 & * & * & * \\
	* & * & * & 0 & 0& 0 \\
	* & * & * & 0 & 0 & 0 \\
	*&  * & * & 0 & 0 & 0 \\
	\end{smallmatrix}
	\right)$ } \\ hammer format}
		\ar[d]
	&&
	*+[F]\txt{\text{$
	\left(
	\begin{smallmatrix} 
	0 & * & * & * & * & *\\
	* & 0 & * & * & * & * \\
	* & * & 0 & * & * & 0 \\
	* & * & * & 0 & * & 0 \\
	* & * & * & * & 0 & 0 \\
	*&  * & 0 & 0 & 0 & 0 \\
	\end{smallmatrix}
	\right)$}\\ arrow format}
	\ar[d]
	\\
	*+[F=]\txt{Proposition \ref{pHammerFormat}}
	&&*+[F=]\txt{Proposition \ref{pArrowFormat}}
}
\caption{Flowchart for the proof of Theorem \ref{tMain}} \label{fFlowThm}
\end{figure}

In this section we prove our main Theorem. The structure of the proof is as follows: If the matrix $M$ has no rank $2$ points the Theorem follows from work of  Manivel and Mezzetti  \cite{MaMe05}. If it does have a rank $2$ point we can prove that $M$ has one of two special formats. For each of these formats we prove the Theorem in a separate proposition. See Figure \ref{fFlowThm}
for a flow chart of this proof.

\begin{proposition} 

\label{pRank2}

Let $M$ be a skew-symmetric $6 \times 6$ matrix with vanishing Pfaffian and 
a rank $2$ point. Then $M$ has one of the following formats
\[
\left(\begin{smallmatrix} 
	0 & * & * & * & * & *\\
	* & 0 & * & * & * & * \\
	* & * & 0 & * & * & * \\
	* & * & * & 0 & 0& 0 \\
	* & * & * & 0 & 0 & 0 \\
	*&  * & * & 0 & 0 & 0 \\
	\end{smallmatrix}\right)
	\quad
\left(\begin{smallmatrix} 
	0 & * & * & * & * & *\\
	* & 0 & * & * & * & * \\
	* & * & 0 & * & * & 0 \\
	* & * & * & 0 & * & 0 \\
	* & * & * & * & 0 & 0 \\
	*&  * & 0 & 0 & 0 & 0 \\
	\end{smallmatrix}\right)
\]
Because of the shape of the zeros, we call the first format {\sl hammer format} and the second one {\sl arrow format}.
\end{proposition}

\begin{proof}
Since there exist a rank $2$ point we can assume that $M$ is of the form
\[
	M = 
	\left(\begin{smallmatrix} 
	0 & a & 0& 0 & 0 & 0\\
	a & 0 & 0& 0 & 0 & 0 \\
	0 & 0 & 0 & 0 & 0 & 0 \\
	0 & 0 & 0 & 0 & 0& 0 \\
	0 & 0 & 0 & 0 & 0 & 0 \\
	0 & 0 & 0 & 0 & 0 & 0 \\
	\end{smallmatrix}\right)
	+ N
\]
with $a \in V^*$ a non zero linear form and $a \not\in \langle n_{01},\dots ,n_{45} \rangle$. By Proposition \ref{pPfaffianLaplace}, it follows
that
\[
	0 = \Pf(M) = a \Pf(N_{01}) + \Pf(N)
\]
and hence $\Pf(N_{01}) = 0$. From the Appendix we have that $N_{01}$ has
one of the following formats
\[
\left(\begin{smallmatrix} 0 & * & * & * \\  * & 0& 0 & 0 \\  * & 0& 0 & 0 \\  * & 0& 0 & 0 \end{smallmatrix}\right) \quad
\left(\begin{smallmatrix}  0 & * & * & 0 \\ * & 0 & * & 0 \\ * & * & 0 & 0   \\ 0 & 0 & 0 & 0   \end{smallmatrix}\right).
\]
Since $N_{01} = M_{01}$ we obtain that $M$ is of format
\[
\left(\begin{smallmatrix} 
	0 & * & * & * & * & *\\
	* & 0 & * & * & * & * \\
	* & * & 0 & * & * & * \\
	* & * & * & 0 & 0& 0 \\
	* & * & * & 0 & 0 & 0 \\
	*&  * & * & 0 & 0 & 0 \\
	\end{smallmatrix}\right)
\]
in the first case and of format
\[
\left(\begin{smallmatrix} 
	0 & * & * & * & * & *\\
	* & 0 & * & * & * & * \\
	* & * & 0 & * & * & 0 \\
	* & * & * & 0 & * & 0 \\
	* & * & * & * & 0 & 0 \\
	*&  * & 0 & 0 & 0 & 0 \\
	\end{smallmatrix}\right)
\]
in the second case.
\end{proof}

We start by analysing matrices that have hammer format:

\begin{proposition} 

\label{pHammerFormat}

Consider a skew $6 \times 6$ matrix $M$ of linear form with hammer format
\[
	\left(\begin{smallmatrix} 
	0 & * & * & * & * & *\\
	* & 0 & * & * & * & * \\
	* & * & 0 & * & * & * \\
	* & * & * & 0 & 0& 0 \\
	* & * & * & 0 & 0 & 0 \\
	*&  * & * & 0 & 0 & 0 \\
	\end{smallmatrix}\right)
\]
and vanishing Pfaffian. Then $M$ has format $(a)$, $(b)$, $(c)$ or $(e)$ in Table \ref{tFormats}.
\end{proposition}	

\begin{proof}
In the situation of the Proposition we have two $3 \times 3$ matrices of linear forms
$N$ and $N'$ such that
\[
	M = \begin{pmatrix} N' & N \\ -N^t & 0 \end{pmatrix}.
\]
We then have
\[
	0 = \Pf(M) = \det(N)
\]
From the Appendix we obtain that $N$ must have one of the following formats
\[
\left(\begin{smallmatrix} * & * & 0 \\  * & * & 0\\ * & * & 0\end{smallmatrix}\right) \quad
\left(\begin{smallmatrix} * & * & * \\  * & * &  *\\ 0 & 0 & 0\end{smallmatrix}\right) \quad
\left(\begin{smallmatrix} * & * & * \\  * & 0 & 0 \\  *& 0 & 0\end{smallmatrix}\right) \quad
\left(\begin{smallmatrix} 0 & * & * \\  * & 0 & * \\ * &*  &0 \end{smallmatrix}\right) 
\]
the last one being skew. In these cases $M$ as format $(a)$, $(b)$, $(c)$ or $(e)$ respectively. 
\end{proof}

\begin{figure}
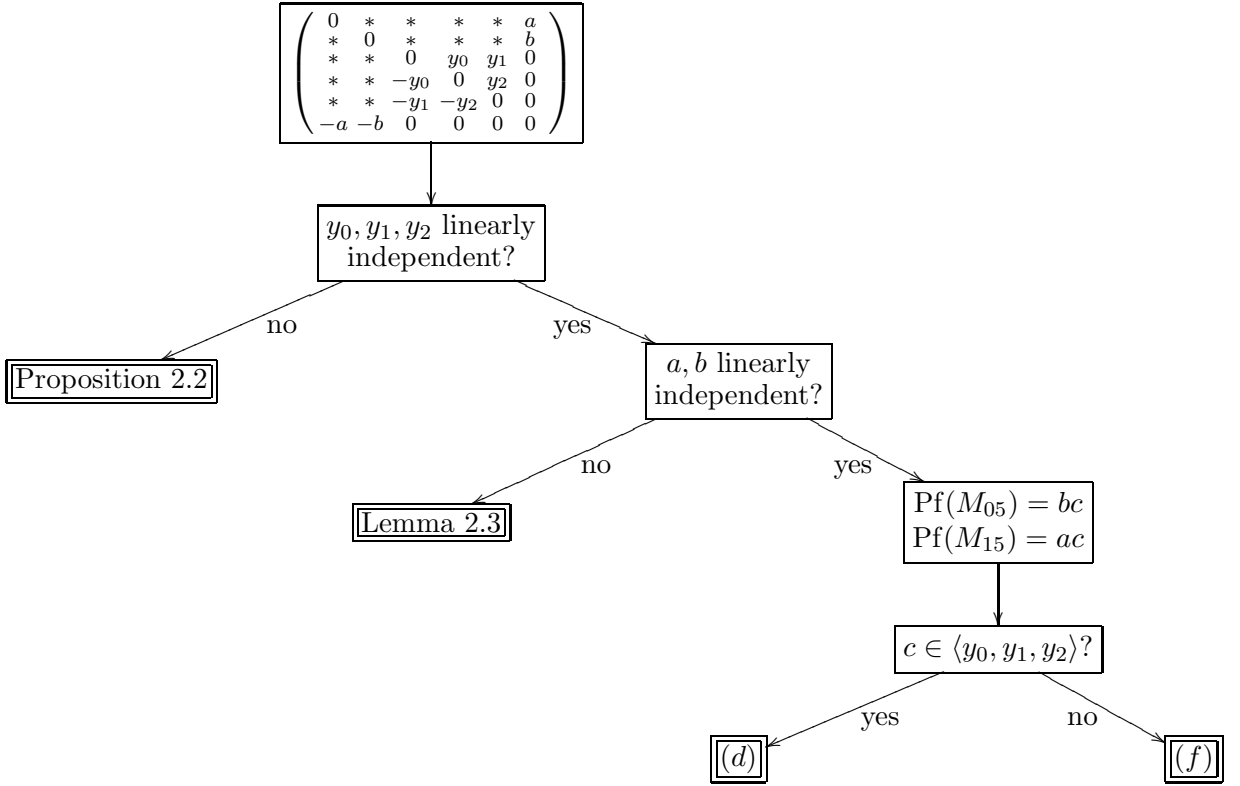

\xycenter{
	&*+[F]\txt{\text{$
	\left(
	\begin{smallmatrix} 
	0 & * & * & * & * & \,\, a \,\,\\
	* & 0 & * & * & * & b\\
	* & * & 0 & y_0 & y_1 & 0 \\
	* & * & -y_0 & 0 & y_2& 0 \\
	* & * & -y_1 & -y_2 & 0 & 0 \\
	-a&  -b & 0 & 0 & 0 & 0 \\
	\end{smallmatrix}
	\right)$}}
	\ar[d]
	\\
	&*+[F]\txt{$y_0, y_1, y_2$ linearly\\ independent?}
	\ar[dl]^{\txt{no}}  
	\ar[dr]_{\txt{yes}} 
	\\
	*+[F=]\txt{Proposition \ref{pHammerFormat}}
	&
	&*+[F]\txt{$a,b$ linearly \\ independent?}
	\ar[dl]^{\txt{no}}
	\ar[dr]_{\txt{yes}}
	&
	\\
	&*+[F=]\txt{Lemma \ref{lNS1}} 
	&
	&*+[F]\txt{$\Pf(M_{05}) = bc$\\ $\Pf(M_{15}) = ac$}
	\ar[d]
	\\
	&&&
	*+[F]\txt{$c \in \langle y_0, y_1, y_2 \rangle$?}
	\ar[dl]^{\txt{yes}}  
	\ar[dr]_{\txt{no}} 
	& 
	\\
	&&*+[F=]\txt{$(d)$} 
	&
	&*+[F=]\txt{$(f)$} 
}

\caption{Flowchart for the proof of Proposition \ref{pArrowFormat}} \label{fFlowArrow}
\end{figure}

This leaves us to look at matrices that have arrow format. The proof is divided into several cases, see Figure \ref{fFlowArrow} for a flow chart of the proof. For future reference we separate one case in the following Lemma:

\begin{lemma} 

\label{lNS1}

Consider a skew $6 \times 6$ matrix $M$ of linear forms with format
\[
	\left(\begin{smallmatrix} 
	0 & * & * & * & * & *\\
	* & 0 & * & * & * & 0 \\
	* & * & 0 & * & * & 0 \\
	* & * & * & 0 & * & 0 \\
	* & * & * & * & 0 & 0 \\
	*&  0 & 0 & 0 & 0 & 0 \\
	\end{smallmatrix}\right)
\]
and vanishing Pfaffian. Then $M$ has format $(a)$, $(b)$ or $(c)$ in Table \ref{tFormats}.
\end{lemma}	

\begin{proof}
We can assume that $M$ is of the above form. If $m_{05} = 0$ then $M$ has format $(a)$ in Table \ref{tFormats}

\medskip

If $m_{05} \not=0$ then $\Pf(M_{05})$ must vanish (Pfaffian Laplace expansion with respect to the last row). By Table \ref{tKnownFormats} there exists a $S \in \GL_4$ such that 
that $S^tM_{05}S$ is of the form
\[
\left(\begin{smallmatrix} 0 & * & * & * \\  * & 0& 0 & 0 \\  * & 0& 0 & 0 \\  * & 0& 0 & 0 \end{smallmatrix}\right) 
\quad \text{or} \quad
\left(\begin{smallmatrix}  0 & * & * & 0 \\ * & 0 & * & 0 \\ * & * & 0 & 0   \\ 0 & 0 & 0 & 0   \end{smallmatrix}\right).
\]
Consequently, after operating with $\left(\begin{smallmatrix} 1 & 0 & 0 \\ 0 & S & 0 \\ 0 & 0 & 1 \end{smallmatrix}\right)$ on $M$ we see that $M$ has format
\[
	\left(\begin{smallmatrix} 
	0 & * & * & * & * & *\\
	* & 0 & * & * & * & 0\\
	* & * & 0 & 0 & 0 & 0 \\
	* & * & 0 & 0 & 0& 0 \\
	* & * & 0 & 0 & 0 & 0 \\
	*&  0 & 0 & 0 & 0 & 0 \\
	\end{smallmatrix}\right)
	\quad \text{or} \quad
	\left(\begin{smallmatrix} 
	0 & * & * & * & * & *\\
	* & 0 & * & * & 0 & 0\\
	* & * & 0 & * & 0& 0 \\
	* & * & * & 0 & 0 & 0 \\
	* & 0 & 0 & 0 & 0 & 0 \\
	*&  0 & 0 & 0 & 0 & 0 \\
	\end{smallmatrix}\right)
\]
i.e. format $(b)$ or $(c)$ in Table \ref{tFormats}.
\end{proof}

\begin{proposition} 

\label{pArrowFormat}

Consider a skew $6 \times 6$ matrix $M$ of linear forms with arrow format
\[
	\left(\begin{smallmatrix} 
	0 & * & * & * & * & *\\
	* & 0 & * & * & * & * \\
	* & * & 0 & * & * & 0 \\
	* & * & * & 0 & * & 0 \\
	* & * & * & * & 0 & 0 \\
	*&  * & 0 & 0 & 0 & 0 \\
	\end{smallmatrix}\right)
\]
and vanishing Pfaffian. Then $M$ has one of the formats in Table \ref{tFormats}.
\end{proposition}	
	 
\begin{proof}
After acting by an element in $\mathrm{GL}_6 (\C )$, we can assume that $M$ is of the following form
\[
\left(\begin{smallmatrix} 
	0 & m_{01} & m_{02} & m_{03} & m_{04} & a \\
	-m_{01} & 0 & m_{12} & m_{13} & m_{14} & b \\
	-m_{02} & -m_{12} & 0 & y_0 & y_1 & 0 \\
	-m_{03} & -m_{13}&  -y_0 & 0 & y_2 & 0 \\
	-m_{04} & -m_{14}&  -y_1 & -y_2 & 0 & 0\\
	-a & -b & 0 & 0 & 0 & 0
	\end{smallmatrix}\right)
\]
Consider the subgroup $H \subset \GL_6$ consisting of matrices of the form
\[
	 \begin{pmatrix}
		\id_2 & S_1 & S_3 \\
		0 & S_2 & S_4\\
		0 & 0 & 1
		\end{pmatrix}
\]
with $S_1$ a $2 \times 3$ matrix, $S_2 \in \GL_3$, $S_3$ a $2 \times 1$ matrix,  and $S_4$ a $3 \times 1$ matrix. $H$ operates 
 on the above $M$ without
changing its form, i.e. $a, b$ and the zeros remain in their places and only $S_2$ operates on 
$\left(\begin{smallmatrix}
	0 & y_0 & y_1  \\
	 -y_0 & 0 & y_2  \\
	  -y_1 & -y_2 & 0 \\
\end{smallmatrix}\right)$.

Notice that we can assume that $y_0,y_1,y_2$ are linearily independent. If not there exists an $h \in H$  such that
$h M h^t$ has the form
\[
	 \left(\begin{smallmatrix} 
	0 & * & * & * & * & *\\
	* & 0 & * & * & * & * \\
	* & * & 0 & * & * & 0 \\
	* & * & * & 0 & 0 & 0 \\
	* & * & * & 0 & 0 & 0 \\
	*&  * & 0 & 0 & 0 & 0 \\
	\end{smallmatrix}\right)
\]
and we are in the situation of Proposition \ref{pHammerFormat}.

\medskip 

\noindent

Notice also that we can assume that $a$ and $b$ are linearly independent in $V^*$. If not, we can 
assume that $M$ is of format
\[
	 \left(\begin{smallmatrix} 
	0 & * & * & * & * & *\\
	* & 0 & * & * & * & 0 \\
	* & * & 0 & * & * & 0 \\
	* & * & * & 0 & 0 & 0 \\
	* & * & * & 0 & 0 & 0 \\
	*&  0 & 0 & 0 & 0 & 0 \\
	\end{smallmatrix}\right)
\]
and we are in the situation of Lemma \ref{lNS1}

\fbox{$\dim \langle a,b \rangle = 2$:} 
Here we have
\[
	0 = \Pf M =  b \Pf(M_{15}) - a \Pf(M_{05}) 
\]
and hence
\[
	 a \Pf(M_{05}) = b \Pf(M_{15}).
\]
This implies that there exist a linear form $c \in V^*$, such that  
\[
	\Pf(M_{05}) = bc \quad \text{and} \quad \Pf(M_{15}) = ac.
\]
We now make a case distinction as to whether $c\in \langle y_0,y_1,y_2 \rangle$ (this includes the case $c=0$)
or $c \not \in \langle y_0,y_1,y_2 \rangle$

\fbox{$\dim \langle a,b \rangle = 2, c\in \langle y_0,y_1,y_2 \rangle$:} 
Write
\[ 
	c = \lambda_0 y_0 + \lambda_1 y_1 + \lambda_2 y_2\
\]
and consider
\[
	S = \id_{\C^6} - 
	\left(\begin{smallmatrix}
	0 & 0 & 0 & 0 & 0  & 0\\
	0 & 0 & 0 & 0 & 0 & 0\\
	0 & 0 & 0 & 0 & 0& \lambda_2 \\
	0 & 0 & 0 & 0 & 0 & -\lambda_1 \\
	0& 0 & 0 & 0 & 0 & \lambda_0\\
	0 & 0 & 0 & 0 & 0 & 0 \\
	\end{smallmatrix}\right)
\]
then
\[
	N = S M S^t = M - 	
	\left(\begin{smallmatrix}
	0 & 0 & a\lambda_2 & -a\lambda_1 &  a\lambda_0 & 0 \\
	0 & 0 & b\lambda_2 & -b\lambda_1 &  b\lambda_0 & 0 \\
	-a\lambda_2 & -b\lambda_2 & 0 & 0 & 0 & 0 \\
	a\lambda_1 & b\lambda_1 & 0 & 0 & 0 & 0 \\
	-a\lambda_0 & -b\lambda_0 &0 & 0 & 0 & 0 \\
	0 & 0 & 0 & 0 & 0  & 0 \\
	\end{smallmatrix}\right)
\]
Hence
\[
	\Pf(N_{05}) = \Pf(M_{05}) - b(\lambda_0y_0+\lambda_1y_1+\lambda_2y_2) = bc - bc = 0
\]
and similarly $\Pf(N_{15}) = 0$, compare \cite{BB22M2}.

Writing out the first Pfaffian we obtain
\[
	0 = n_{14}y_0-n_{13}y_1+n_{12}y_2 = (y_0,-y_1,y_2)
	\begin{pmatrix} n_{14} \\ n_{13} \\ n_{12} \end{pmatrix},
\]	
i.e $(n_{14}, n_{13}, n_{14})^t$ is a linear syzygy of $(y_0,-y_1,y_2)$. The vector space of all linear syzygies of $(y_0,-y_1,y_2)$ is generated by the columns of
\[
\begin{pmatrix}
       0&{y}_{2}&{y}_{1}\\
       {-{y}_{2}}&0&{y}_{0}\\
       {-{y}_{1}}&{-{y}_{0}}&0\end{pmatrix}.
\]
Hence there exist constants $\alpha_0,\alpha_1,\alpha_2 \in \C$ such that 
\[
N = \left(\begin{smallmatrix} 
	0 & n_{01} & n_{02} & n_{03} & n_{04} & a \\
	-n_{01} & 0 & -\alpha_0y_1-\alpha_1y_0  &  -\alpha_0y_2+\alpha_2y_0 & \alpha_1y_2+\alpha_2y_1& b \\
	-n_{02} & \alpha_0y_1+\alpha_1y_0 & 0 & y_0 & y_1 & 0 \\
	-n_{03} & \alpha_0y_2-\alpha_2y_0&  -y_0 & 0 & y_2 & 0 \\
	-n_{04} & -\alpha_1y_2-\alpha_2y_1&  -y_1 & -y_2 & 0 & 0\\
	-a & -b & 0 & 0 & 0 & 0
	\end{smallmatrix}\right)
\]
We see that we can operate with an appropriate $S \in H$ to obtain
\[
\left(\begin{smallmatrix} 
	0 & n_{01}' & n_{02} & n_{03} & n_{04} & a \\
	-n_{01}' & 0 & 0 & 0 & 0 & b \\
	-n_{02} & 0& 0 & y_0 & y_1 & 0 \\
	-n_{03} & 0&  -y_0 & 0 & y_2 & 0 \\
	-n_{04} & 0&  -y_1 & -y_2 & 0 & 0\\
	-a & -b & 0 & 0 & 0 & 0
	\end{smallmatrix}\right)
\]
using $\Pf(N_{15})=0$ in the same way we obtain
\[
\left(\begin{smallmatrix} 
	0 & n_{01}' & 0 & 0 & 0& a \\
	-n_{01}' & 0 & 0 & 0 & 0 & b \\
	0 & 0& 0 & y_0 & y_1 & 0 \\
	0 & 0&  -y_0 & 0 & y_2 & 0 \\
	0 & 0&  -y_1 & -y_2 & 0 & 0\\
	-a & -b & 0 & 0 & 0 & 0
	\end{smallmatrix}\right).
\]
After rotating the last $4$ rows  and columns cyclically, we see that $N$ and hence $M$ is of format $(d)$.

\fbox{$\dim \langle a,b \rangle = 2, c\not\in \langle y_0,y_1,y_2 \rangle$:} 
Choosing a basis of $V^*$ that contains $y_0$, $y_1$, $y_2$ and $c$, we 
let 
\newcommand{\barm}{\overline{m}}
\[
\overline{M} = \left(\begin{smallmatrix} 
	0 & \barm_{01} & \barm_{02} & \barm_{03} & \barm_{04} & \overline{a} \\
	-\barm_{01} & 0 & \barm_{12} & \barm_{13} & \barm_{14} & \overline{b} \\
	-\barm_{02} & -\barm_{12} & 0 & y_0 & y_1 & 0 \\
	-\barm_{03} & -\barm_{13}&  -y_0 & 0 & y_2 & 0 \\
	-\barm_{04} & -\barm_{14}&  -y_1 & -y_2 & 0 & 0\\
	-\overline{a} & -\overline{b} & 0 & 0 & 0 & 0
	\end{smallmatrix}\right)
\]
be the result of substituting $c=0$ into $M$.

We then have
\[
	\Pf(\overline{M}_{05}) = \overline{b}\overline{c} =0 \quad \text{and}  \quad \Pf(\overline{M}_{15}) = \overline{a}\overline{c} =0
\]
As in the previous case we can find an $S \in H$ such that 
\[
S \overline{M} S^t = \left(\begin{smallmatrix} 
	0 & * & 0 & 0 & 0 & \overline{a} \\
	* & 0 & 0& 0& 0 & \overline{b} \\
	0 & 0 & 0 & y_0 & y_1 & 0 \\
	0 & 0&  -y_0 & 0 & y_2 & 0 \\
	0 & 0&  -y_1 & -y_2 & 0 & 0\\
	-\overline{a} & -\overline{b} & 0 & 0 & 0 & 0
	\end{smallmatrix}\right)
\]
Applying the same $S$ to $M$ gives
\[
M' = S M S^t = \left(\begin{smallmatrix} 
	0 & m_{01}'& \lambda_{02}c & \lambda_{03}c & \lambda_{04}c & a \\
	-m_{01}' & 0 & \lambda_{12}c& \lambda_{13}c& \lambda_{14}c & b \\
	-\lambda_{02}c & -\lambda_{12}c & 0 & y_0 & y_1 & 0 \\
	-\lambda_{03}c  & -\lambda_{13}c &  -y_0 & 0 & y_2 & 0 \\
	-\lambda_{04}c  & -\lambda_{14}c &  -y_1 & -y_2 & 0 & 0\\
	-a& -b & 0 & 0 & 0 & 0
	\end{smallmatrix}\right)
\]
Computing  the Pfaffian gives
\[
	0=\Pf(M') = -c\bigl((\lambda_{14}{y}_{0}-\lambda_{13}{y}_{1}+\lambda_{12}{y}_{2})a-(\lambda_{04}{y}_{0}-\lambda_{03}{y}_{1}+\lambda_{02}{y}_{2})b\bigr)
\]
compare \cite{BB22M2}. Since $c \not=0$
\[
	(\lambda_{14}{y}_{0}-\lambda_{13}{y}_{1}+\lambda_{12}{y}_{2})a=(\lambda_{04}{y}_{0}-\lambda_{03}{y}_{1}+\lambda_{02}{y}_{2})b
\]
It follows that $a$ and $b$ are proportional to $(\lambda_{04}{y}_{0}-\lambda_{03}{y}_{1}+\lambda_{02}{y}_{2})$
and $(\lambda_{14}{y}_{0}-\lambda_{13}{y}_{1}+\lambda_{12}{y}_{2})$, i.e $a,b \in \langle y_0,y_1,y_2\rangle$. We can therefore assume that after another $H$-operation we have $y_0=-a$ and $y_1=-b$. The equation above is then
\[
-\lambda_{14}a^2+\lambda_{13}ab+\lambda_{12}a{y}_{2}=-\lambda_{04}ab+\lambda_{03}b^2+\lambda_{02}b{y}_{2}
\]
and, by comparing coefficients we get 
\[
\lambda_{14}=\lambda_{12}=\lambda_{03}=\lambda_{02} = 0 \text{ and } \lambda_{13} = - \lambda_{04} =: -\lambda.
\]
This shows
\[
M' =  \left(\begin{smallmatrix} 
	0 & m_{01}'& 0 & 0 & \lambda c &  a \\
	-m_{01}' & 0 & 0 & -\lambda c & 0 & b \\
	0 & 0 & 0 & -a & -b & 0 \\
	0 & \lambda c & a & 0 & y_2 & 0 \\
	-\lambda c & 0 &  b & -y_2 & 0 & 0\\
	-a& -b & 0 & 0 & 0 & 0
	\end{smallmatrix}\right)
\]
We see that $M'$ and hence $M$ is of format $(f)$ in Table \ref{tFormats}.
\end{proof}

Now we can prove:

\begin{theorem} 

\label{tMain}

Consider a skew $6 \times 6$ matrix $M$ of linear forms with vanishing Pfaffian.
Then $M$ has one of the formats in Table \ref{tFormats}.
\end{theorem}

\begin{proof}
If $M$ has empty rank $2$ locus, then the Theorem follows from work of Manivel and Mezzetti \cite[Proposition 2, Theorem 4]{MaMe05}. Indeed their cases $\ell_s$ and $\ell_g$ have format $(b)$, their case $\pi_5$ has format $(a)$, their case $\pi_t$ has format $(b)$, their case $\pi_p$ has format $(c)$ and their case $\pi_g$ has format $(d)$.

If $M$ does have a rank $2$ point, we know from Proposition \ref{pRank2} that $M$ has one of the following formats
\[
\left(\begin{smallmatrix} 
	0 & * & * & * & * & *\\
	* & 0 & * & * & * & * \\
	* & * & 0 & * & * & * \\
	* & * & * & 0 & 0& 0 \\
	* & * & * & 0 & 0 & 0 \\
	*&  * & * & 0 & 0 & 0 \\
	\end{smallmatrix}\right)
	\quad
\left(\begin{smallmatrix} 
	0 & * & * & * & * & *\\
	* & 0 & * & * & * & * \\
	* & * & 0 & * & * & 0 \\
	* & * & * & 0 & * & 0 \\
	* & * & * & * & 0 & 0 \\
	*&  * & 0 & 0 & 0 & 0 \\
	\end{smallmatrix}\right)
\]
In the first case the claim follows from Proposition \ref{pHammerFormat}, in the second case from Proposition \ref{pArrowFormat}.
\end{proof}

\section{Strictly semistable matrices}

We consider the operation of $\GL_6(\C)$ on the space of skew symmetric $6 \times 6$ matrices of linear forms given by
\[
	M \mapsto S M S^t.
\]
Gaia Comaschi in \cite[Thm. 3.1]{Co21} has given a criterion in terms of formats to check when a skew-symmetric $6 \times 6$ matrix of linear forms is not stable or not semistable with respect to this action: 

\begin{theorem}\label{tComaschi}
Let $M$ be a skew-symmetric $6 \times 6$ matrix of linear forms. Then $M$ is not stable, therefore strictly semistable or unstable,  if and only if $M$ has one of the following skew formats: 

\[
(\mathrm{NS}_1): 
\left(\begin{smallmatrix}
0 & 0 & 0 & 0 & 0 &* \\
0 & 0 & * & * & * & *  \\
0 & * & 0 & * & * & * \\
0 & * & * & 0 & * & * \\
0 & * & * & *  & 0 & * \\
* & * & * & * & * & 0 
\end{smallmatrix}\right) 
\quad
(\mathrm{NS}_2): 
\left(\begin{smallmatrix}
0 & 0 & 0 & 0 & * &* \\
0 & 0 & 0 & 0 & * & *  \\
0 & 0 & 0 & * & * & * \\
0 & 0 & * & 0 & * & * \\
* & * & * & *  & 0 & * \\
* & * & * & * & * & 0 
\end{smallmatrix}\right) 
\]

\[
(\mathrm{NS}_3): 
\left(\begin{smallmatrix}
0 & 0 & 0 & * & * &* \\
0 & 0 & 0 & * & * & *  \\
0 & 0 & 0 & * & * & * \\
* & * & * & 0 & * & * \\
* & * & * & *  & 0 & * \\
* & * & * & * & * & 0 
\end{smallmatrix}\right) 
\]

Furthermore, $M$  is not semistable, or equivalently, unstable, if and only if it has one of the following skew formats: 

\[
(\mathrm{NSS}_1): 
\left(\begin{smallmatrix}
0 & 0 & 0 & 0 & 0 & 0 \\
0 & 0 & * & * & * & *  \\
0 & * & 0 & * & * & * \\
0 & * & * & 0 & * & * \\
0 & * & * & *  & 0 & * \\
0 & * & * & * & * & 0 
\end{smallmatrix}\right) 
\quad
(\mathrm{NSS}_2): 
\left(\begin{smallmatrix}
0 & 0 & 0 & 0 & 0 &* \\
0 & 0 & 0 & 0 & 0 & *  \\
0 & 0 & 0 & * & * & * \\
0 & 0 & * & 0 & * & * \\
0 & 0 & * & *  & 0 & * \\
* & * & * & * & * & 0 
\end{smallmatrix}\right) 
\]

\[
(\mathrm{NSS}_3): 
\left(\begin{smallmatrix}
0 & 0 & 0 & 0 & * &* \\
0 & 0 & 0 & 0 & * & *  \\
0 & 0 & 0 & 0 & * & * \\
0 & 0 & 0 & 0 & * & * \\
* & * & * & *  & 0 & * \\
* & * & * & * & * & 0 
\end{smallmatrix}\right)
\]
\end{theorem}

Notice that Comaschi uses the convention of collecting the zeros in the upper left corner, while we (following Manivel and Mezzetti \cite{MaMe05}) collect the zeros in the lower right corner. To convert between these conventions we say that the $r \times r$ matrix
\[
	 M^{\mathrm{rev}} = (m_{r-i-1,r-j-1})_{0 \le i,j < r}
\]
is obtained from $M = (m_{i,j})_{0 \le i,j < r}$ by {\sl reversing rows and columns}. We also call $ M^{\mathrm{rev}}$ the {\sl reverse matrix}. Notice that if $M$ has format $F$ it also has the format described by the reverse of $F$.

We  now assume in addition that the Pfaffian of $M$ vanishes. If $M$ is stable it has to have format $(d)$ or $(f)$ in Table \ref{tFormats} since $(a)$, $(b)$, $(c)$ are unstable by the above Theorem and matrices of format $(e)$ are not stable.  Comaschi gives normal forms for the stable cases. At the end of this section we reprove her results using our methods. Before that we extend Comaschi's classification to the strictly semistable case. We start by narrowing down the possible format of a strictly semistable matrix $M$:


\begin{figure}
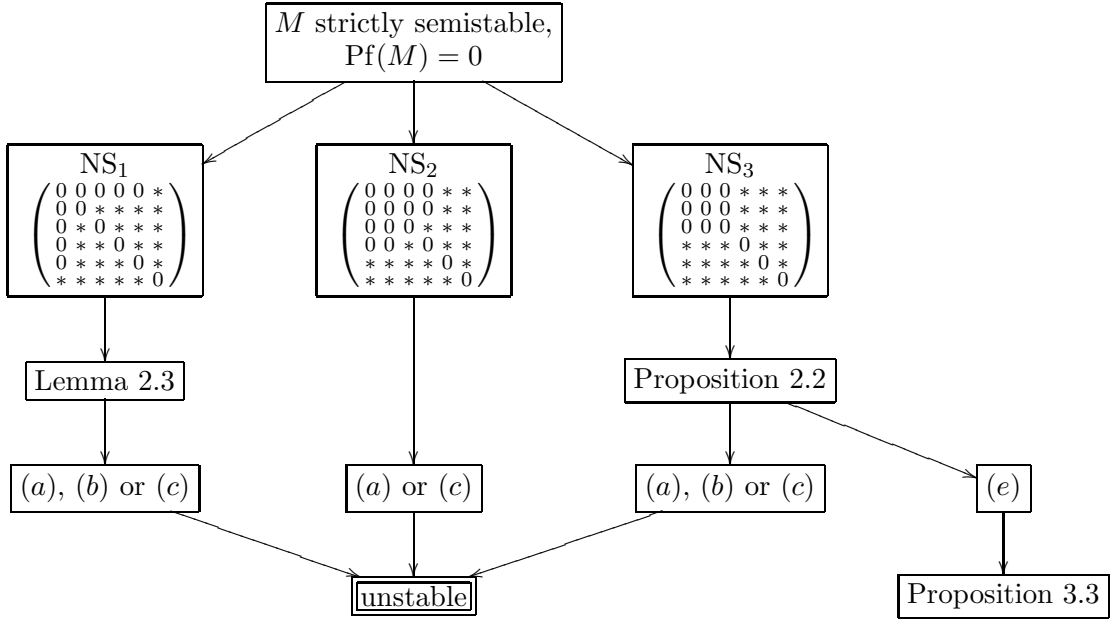

\xycenter{
	&*+[F]\txt{$M$ strictly semistable, \\ $\Pf(M) = 0$}
	\ \ar[dl]  \ar[d] \ar[dr]
	\\
	*+[F]\txt{$\mathrm{NS}_1$ \\ \text{
	$\left(\begin{smallmatrix}
	0 & 0 & 0 & 0 & 0 &* \\
	0 & 0 & * & * & * & *  \\
	0 & * & 0 & * & * & * \\
	0 & * & * & 0 & * & * \\
	0 & * & * & *  & 0 & * \\
	* & * & * & * & * & 0 
	\end{smallmatrix}\right)$}
	}
	\ar[d]
	&
	*+[F]\txt{$\mathrm{NS}_2$ \\ \text{
	$\left(\begin{smallmatrix}
	0 & 0 & 0 & 0 & * &* \\
	0 & 0 & 0 & 0 & * & *  \\
	0 & 0 & 0 & * & * & * \\
	0 & 0 & * & 0 & * & * \\
	* & * & * & *  & 0 & * \\
	* & * & * & * & * & 0 
	\end{smallmatrix}\right)$}
	}
	\ar[dd]
	&
	*+[F]\txt{$\mathrm{NS}_3$ \\ \text{
	$\left(\begin{smallmatrix}
	0 & 0 & 0 & * & * &* \\
	0 & 0 & 0 & * & * & *  \\
	0 & 0 & 0 & * & * & * \\
	* & * & * & 0 & * & * \\
	* & * & * & *  & 0 & * \\
	* & * & * & * & * & 0 
	\end{smallmatrix}\right)$}
	}
	\ar[d]
	\\
	*+[F]\txt{Lemma \ref{lNS1}}
	\ar[d]
	&&
	*+[F]\txt{Proposition \ref{pHammerFormat}}
	\ar[d] \ar[dr]
	\\
	*+[F]\txt{$(a)$, $(b)$ or $(c)$}
	\ar[dr]
	&*+[F]\txt{$(a)$  or $(c)$}
	\ar[d]
	&*+[F]\txt{$(a)$, $(b)$  or $(c)$}
	\ar[dl]
	&*+[F]\txt{$(e)$}
	\ar[d]
	\\
	&*+[F=]\txt{unstable}	
	&
	&*+[F]\txt{Proposition \ref{pSemiStable}}		
}

\caption{Flowchart for the classification of strictly semistable matrices  $M$ with $\Pf(M) = 0$} \label{fFlowSemiStable}
\end{figure}


\begin{proposition} \label{pSemiStableDoubleSkew}
Let $M$ be a strictly semistable, skew symmetric $6 \times 6$ matrix of linear forms with vanishing Pfaffian. Then $M$ has double skew format
\[
\left(\begin{smallmatrix}
0 & 0 & 0 & 0 & *  & *\\
0 & 0 & 0 & * & 0 & *\\
0 & 0 & 0 & * &  *& 0\\
0 & * & *  & 0 & *& *\\
 * & 0 & * & * & 0 & * \\
* & *& 0 & * & * & 0 \\
\end{smallmatrix}\right).
\]
\end{proposition}

\begin{proof}
By Theorem \ref{tComaschi} we know that $M$ has format $NS_1$, $NS_2$ or $NS_3$.  By assumption $\Pf(M) =0$. 
We treat each case in turn:

\fbox{$NS_1$:} 
If $M$ has format $NS_1$ we can reverse rows and columns and apply Lemma \ref{lNS1}. It follows that $M$ has format $(a)$, $(b)$ or $(c)$ in Table \ref{tFormats}. But then format $(a)$ is the reverse of $NSS_1$, format $(b)$ is reverse of $NSS_3$ and format $(c)$ is the reverse $NSS_2$. Therefore $M$ must be unstable. 

\fbox{$NS_2$:} 
We can assume, that $M$ is of the form
\[
	M = \begin{pmatrix}
		0 & 0 & C \\
		0 & D & B \\
		-C^t & -B^t & A
		\end{pmatrix}
\]
with $2 \times 2$ matrices $A,B,C,D$ of which $A$ and $D$ are skew. We then have
\[
	0 = \Pf(M) = \det(C) \cdot \Pf(D) 
\]
If $\Pf(D) = 0$ then $D=0$ since $D$ is a $2 \times 2$ matrix. But then $M$ has format $NSS_3$ and is therefore unstable. 
This leaves the case $\det(C) =0$. By reversing the formats in Table \ref{tKnownFormats} we know that the matrix $C$ has format
\[
\left(\begin{smallmatrix} 0 &  * \\  0 & *\end{smallmatrix}\right) \quad \text{or} \quad
\left(\begin{smallmatrix} 0 & 0 \\  * & *\end{smallmatrix}\right).
\]
In the first case $M$ has format $NSS_2$, in the second case $NSS_1$. Again $M$ must be unstable.

\fbox{$NS_3$:} 
In this case $M$ has hammer format and $\Pf(M) = 0$, so we are (after reversing rows and columns of $M$) in the situation of Proposition \ref{pHammerFormat}. It follows that $M$ has format  $(a)$, $(b)$, $(c)$ or $(e)$ in Table \ref{tFormats}. $(a)$, $(b)$ and $(c)$ are again unstable, so that $M$ must have format double skew format $(e)$. Reversing format $(e)$ we obtain the claim. 
\end{proof}

\begin{proposition} \label{pSemiStable}
Let $M$ be a strictly semistable, skew-symmetric $6 \times 6$ matrix of linear forms with vanishing Pfaffian. Let
$r = \dim \P \langle m_{01},\dots,m_{45} \rangle$ the dimension of the projective space spanned by the entries of $M$.
Then $r \in \{2,3,4,5\}$ and there exist a matrix $S \in \GL_6(\C)$ and linear forms $x_0,\dots,x_5$ such that 
\[
S^tMS = \left(\begin{smallmatrix}
0 & 0 &  0 & 0 & x_0  & x_1\\
0 & 0 & 0 & -x_0 & 0 & x_2\\
 0 & 0 & 0 & -x_1 &  -x_2 & 0\\
 0 & x_0 & x_1  & 0 & x_3 & x_4\\
 -x_0 & 0 & x_2 & -x_3& 0 & x_5 \\
-x_1 & x_2 & 0 & -x_4 & -x_5 & 0 \\
\end{smallmatrix}\right)
\]
with $x_0,\dots,x_r$ linearly independent and $x_i = 0$ for $i > r$.
\end{proposition}

\begin{proof}
We write $M$ as
\[
	M = \begin{pmatrix} 0 & B \\ -B^t & A \end{pmatrix}
\]
with $A,B$ skew-symmetric $3 \times 3$ matrices of linear forms. If the entries of $B$ are linearly dependent, $B$ has skew format
\[
\left(\begin{smallmatrix}
 0 & 0 & *  \\
 0 & 0 & * \\
* & * & 0  \\
\end{smallmatrix}\right).
\]
and hence $M$ has format $NSS_2$. So the entries of $B$ must be linearly independent.

\medskip 

\noindent
We can now write $A = A' + A''$ such that  
\begin{enumerate}
\item $A'$ and $A''$ are skew-symmetric $3 \times 3$ matrices of linear forms.
\item $\langle A' \rangle \subset \langle B \rangle$, i.e the entries of $A'$ are in the span of the entries of $B$.
\item $\langle A'' \rangle \cap \langle B \rangle = \{0\}$, i.e. the entries of $A''$ are independent of the entries of $B$. 
\end{enumerate}
Because of condition $b)$ there exists a not necessarily invertible $3 \times 3$ matrix $T$ such that
\[
	TB- B^tT^t = A'.
\]
Setting
\[
	S = \begin{pmatrix} \id & 0\\ -T & \id \end{pmatrix}
\]
we have
\[
	S M S^t = \begin{pmatrix}
			0 & B \\
			-B^t & A - TB + B^tT^t 
			\end{pmatrix} 
			=
			\begin{pmatrix}
			0& B \\
			-B^t & A''
			\end{pmatrix} 
\]
We can therefore choose a basis of $x_3,\dots,x_r$ of $\langle A'' \rangle$ and assume that after another $\GL_6(\C)$ operation 
\[
	A''  = \left(\begin{smallmatrix}
 0 & x_3 & x_4  \\
 -x_3 & 0 & x_5 \\
-x_4 & -x_5 & 0  \\
\end{smallmatrix}\right).
\]
with $x_i = 0$ for $i>r$. We now set $x_0=b_{01}, x_1= b_{02}$ and $x_2 = b_{12}$, i.e
\[
	B = \left(\begin{smallmatrix}
 0 & x_0 & x_1  \\
 -x_0 & 0 & x_2 \\
-x_1 & -x_2 & 0  \\
\end{smallmatrix}\right).
\]
Since the entries of $B$ are linearly independent and $\langle B \rangle \cap \langle A'' \rangle  = \{0\}$ we
have that $x_0,\dots, x_r$ are also linearly independent and $S^tMS$ has the claimed form. 
\end{proof}

\begin{figure}
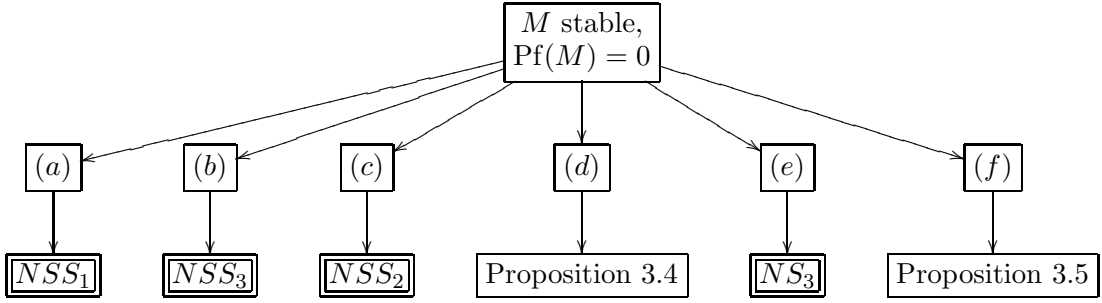

\xycenter{
	&&&*+[F]\txt{$M$ stable, \\$\Pf(M) = 0$} 
	\ar[dlll] \ar[dll] \ar[dl]  \ar[d] \ar[dr] \ar[drr]
	\\
	*+[F]\txt{$(a)$} \ar[d]
	&*+[F]\txt{$(b)$} \ar[d]
	&*+[F]\txt{$(c)$} \ar[d]
	&*+[F]\txt{$(d)$} \ar[d]
	&*+[F]\txt{$(e)$} \ar[d]
	&*+[F]\txt{$(f)$} \ar[d]
	\\
	*+[F=]\txt{$NSS_1$}
	&*+[F=]\txt{$NSS_3$}
	&*+[F=]\txt{$NSS_2$}
	&*+[F]\txt{Proposition \ref{pStableD}}
	&*+[F=]\txt{$NS_3$}
	&*+[F]\txt{Proposition \ref{pStableF}}
	\\
}
\caption{Flowchart for the classification of stable matrices  $M$ with $\Pf(M) = 0$} \label{fFlowStable}
\end{figure}

\begin{proposition} \label{pStableD}
If $M$ is a stable skew-symmetric $6 \times 6$ matrix of linear forms of format $(d)$ then
after an $\GL_6(\C)$ operation we have
\[
M =  \left(
\begin{smallmatrix}
      0&{l}_{0}&{l}_{1}\\
      {-{l}_{0}}&0&{l}_{2}\\
      {-{l}_{1}}&{-{l}_{2}}&0\\
      &&&0&{m}_{0}&{m}_{1}\\
      &&&{-{m}_{0}}&0&{m}_{2}\\
      &&&{-{m}_{1}}&{-{m}_{2}}&0
\end{smallmatrix}
\right)
\]
with $\dim \langle l_0,l_1,l_2 \rangle = \dim \langle m_0,m_1,m_2 \rangle=3$.
\end{proposition}
\begin{proof}
If $M$ is of format $(d)$ we have 
\[
M =  \left(
\begin{smallmatrix}
      0&{l}_{0}&{l}_{1}\\
      {-{l}_{0}}&0&{l}_{2}\\
      {-{l}_{1}}&{-{l}_{2}}&0\\
      &&&0&{m}_{0}&{m}_{1}\\
      &&&{-{m}_{0}}&0&{m}_{2}\\
      &&&{-{m}_{1}}&{-{m}_{2}}&0
\end{smallmatrix}
\right)
\]
with as yet no information about linear independence of the entries. If $\dim \langle l_0,l_1,l_2 \rangle < 3$ we have,
can obtain, after a $\GL_6(\C)$-operation
\[
M =\left(\begin{smallmatrix}
      0&{l}_{0}&{l}_{1}\\
      {-{l}_{0}}&0&0\\
      {-{l}_{1}}&0&0\\
      &&&0&{m}_{0}&{m}_{1}\\
      &&&{-{m}_{0}}&0&{m}_{2}\\
      &&&{-{m}_{1}}&{-{m}_{2}}&0
\end{smallmatrix}
\right)
\]
which is unstable because it has format $NSS_2$.
\end{proof}

\begin{proposition} \label{pStableF}
If $M$ is a stable skew-symmetric $6 \times 6$ matrix of linear forms of format $(f)$ that is not of format $(d)$, then
after a $\GL_6(\C)$-operation we have
\[
M =  \left(
 \begin{smallmatrix}
       &{l}_{3}&&0&{l}_{0}&{l}_{1}\\
      {-{l}_{3}}&&&{-{l}_{0}}&0&{l}_{2}\\
      &&&{-{l}_{1}}&{-{l}_{2}}&0\\
      0&{l}_{0}&{l}_{1}&&{l}_{4}&\\
      {-{l}_{0}}&0&{l}_{2}&{-{l}_{4}}&&\\
      {-{l}_{1}}&{-{l}_{2}}&0&&&
\end{smallmatrix}
\right)
\]
with $l_0,\dots, l_4$ linearly independent.
\end{proposition}

\begin{proof}
Since $M$ is of format $(f)$ we can assume
\[
M =  \left(
 \begin{smallmatrix}
       &{l}_{3}&&0&{l}_{0}&m\\
      {-{l}_{3}}&&&{-{l}_{0}}&0&n\\
      &&&{-m}&{-n}&0\\
      0&{l}_{0}&m&&{l}_{4}&\\
      {-{l}_{0}}&0&n&{-{l}_{4}}&&\\
      {-m}&{-n}&0&&&
\end{smallmatrix}
\right)
\]
If $m$ and $n$ are linearly dependent we are in the situation of Lemma $\ref{lNS1}$ and $M$ is unstable. 
We can therefore assume that $m$ and $n$ are linearly independent. 

Applying the permutation $(0,1,2,3,4,5) \mapsto (0,3,1,4,5,2)$ to rows and columns we can write 
$M$ as a $3 \times 3$ matrix of symmetric $2 \times 2$ matrices
\[
	M = \begin{pmatrix}
		0 & A & m \cdot\id \\
		-A & 0 & n\cdot\id \\
		-m\cdot\id & -n\cdot\id & 0
		\end{pmatrix}
\]
with $A = \left(\begin{smallmatrix} l_3 & l_0 \\ l_0 & l_4 \end{smallmatrix} \right)$ a \emph{symmetric} $2\times 2$-matrix \cite{BB22M2}. We now write
\[
	A = A' + mA'' + nA'''
\]
with $A''$ and $A'''$ symmetric $2 \times 2$ matrices over $\C$ and $A'$ a symmetric $2 \times 2$ matrix of 
linear forms with $\langle A' \rangle \cap \langle m,n \rangle = \{0\}$. With
\[
	S = \begin{pmatrix}
		\id & 0 & A''' \\
		0 & \id & -A'' \\
		0 & 0 & \id
	\end{pmatrix}
\]
we have
\begin{align*}
M' = S M S^t &= \begin{pmatrix}
		0 & A - mA'' - nA'''& m \cdot \id \\
		-(A-mA''-nA''')& 0 & n\cdot id \\
		-m\cdot\id & -n\cdot\id & 0
		\end{pmatrix}\\
	     &= \begin{pmatrix}
		0 & A ' & m\cdot\id \\
		-A' & 0 & n\cdot\id \\
		-m\cdot\id & -n\cdot\id & 0
		\end{pmatrix} 
\end{align*}
compare \cite{BB22M2}. 
Assume now to the contrary
\[
5 > \dim \langle M \rangle = \dim \langle m,n \rangle + \dim \langle A' \rangle = 2 + \dim \langle A' \rangle
\]
and therefore $\dim \langle A' \rangle < 3$. We then have a matrix $T \in \GL_2(\C)$ such that
$T A' T^t$ is 
\[
	\left( \begin{smallmatrix} a & b \\ b & 0 \end{smallmatrix} \right)
	\quad \text{or} \quad
	\left( \begin{smallmatrix} a & 0\\ 0 & b \end{smallmatrix} \right)
\]
with $a, b$ possibly zero linear forms. 
Operating with 
\[
	\begin{pmatrix}
	 T^t & 0 & 0 \\
	 0 & T^t & 0 \\
	 0 & 0 & T^{-1}
	 \end{pmatrix}
\]
we see that we can assume that $A'$ is 
$\left(\begin{smallmatrix} a & b \\ b & 0 \end{smallmatrix} \right)$
or
$\left( \begin{smallmatrix} a & 0\\ 0 & b \end{smallmatrix} \right)$. 
We apply the reverse of the permutation above. In the first case we obtain
\[
\left(
 \begin{smallmatrix}
       &{a}&&0&b&m\\
      -{a}&&&{-b}&0&n\\
      &&&{-m}&{-n}&0\\
      0&b&m&&0&\\
      {-b}&0&n&&&\\
      {-m}&{-n}&&&&
\end{smallmatrix}
\right)	
\]
and $M$ is not stable by $NS_3$. In the second case we obtain
\[
\left(
 \begin{smallmatrix}
       0&{a}&0&0&0&m\\
      -{a}&0&0&0&0&n\\
      0&0&0&{-m}&{-n}&0\\
      0&0&m&0&b&0\\
      0&0&n&-b&0&0\\
      {-m}&{-n}&0&0&0&0
\end{smallmatrix}
\right)	
\]
rearranging the rows and columns again, we get
\[
\left(
 \begin{smallmatrix}
       0&{a}&m&0&0&0\\
      -{a}&0&n&0&0&0\\
      -m&-n&0&0&0&0\\
      0&0&0&0&{-m}&{-n}\\
      0&0&0&m&0&b\\
      0&0&0&n&-b&0\\
\end{smallmatrix}
\right)	
\]	
and $M$ has format $(d)$, contrary to our assumption. 
We therefore must have $\dim \langle M \rangle = 5$ as claimed. 
\end{proof}

\section{$\P^4$}

For applications semistable $6 \times 6$ matrices of linear forms over $\P^4$ with vanishing Pfaffian seem to be particularly interesting. We collect here the normal forms of such matrices for future reference.

\newcommand{\Ytop}{Y^{\mathrm{top}}}

\begin{table}
\begin{tabular}{|c|c|c|c| c|}
\hline
& $M$ & $\Ytop$ & $d_4$ &  \text{stability} \\
\hline
 (a)
 &
  $\left(
 \begin{smallmatrix}
       &{l}_{3}&&0&{l}_{0}&{l}_{1}\\
      {-{l}_{3}}&&&{-{l}_{0}}&0&{l}_{2}\\
      &&&{-{l}_{1}}&{-{l}_{2}}&0\\
      0&{l}_{0}&{l}_{1}&&{l}_{4}&\\
      {-{l}_{0}}&0&{l}_{2}&{-{l}_{4}}&&\\
      {-{l}_{1}}&{-{l}_{2}}&0&&&
\end{smallmatrix}
\right)$
&
a smooth conic
&
10
&
stable
\\ \hline
 (b)
 &
$\left(
\begin{smallmatrix}
      0&{l}_{0}&{l}_{1}\\
      {-{l}_{0}}&0&{l}_{2}\\
      {-{l}_{1}}&{-{l}_{2}}&0\\
      &&&0&{l}_{2}&{l}_{3}\\
      &&&{-{l}_{2}}&0&{l}_{4}\\
      &&&{-{l}_{3}}&{-{l}_{4}}&0
\end{smallmatrix}
\right)$
&
two skew lines
&
9
&
 stable
\\ \hline
 (c)
 &
$\left(
\begin{smallmatrix}
      0&{l}_{0}&{l}_{1}\\
      {-{l}_{0}}&0&{l}_{2}\\
      {-{l}_{1}}&{-{l}_{2}}& 0\\
      &&&0&{l}_{1}&{l}_{2}\\
      &&&{-{l}_{1}}&0&{l}_{3}\\
      &&&{-{l}_{2}}&{-{l}_{3}}&0
\end{smallmatrix}
\right)$
&
 \begin{tabular}{c}
two distinct \\ intersecting lines
\end{tabular}
&
8
&
stable
\\ \hline
 (d)
 &
$\left(
\begin{smallmatrix}
&  &  & 0 &l _0 &l_1 \\
 &  &  & -l_0 & 0 & l_2  \\
 &  &  & -l_1 & -l_2 & 0 \\
0 &l _0 &l_1 & & l_3 & l_4\\
-l_0 & 0 & l_2  & -l_3 &  & \\
-l_1 & -l_2 & 0 & -l_4 &  & 
\end{smallmatrix}
\right)$
&
  \begin{tabular}{c}
  a double line lying on \\
a smooth quadric surface\\
\end{tabular}
&
9
&
\begin{tabular}{c}
strictly semistable,\\ but not polystable
\end{tabular}
\\ \hline
(e)
 &
$\left(
\begin{smallmatrix}
&  &  & 0 &l _0 &l_1 \\
 &  &  & -l_0 & 0 & l_2  \\
 &  &  & -l_1 & -l_2 & 0 \\
0 &l _0 &l_1 & & l_3 & \\
-l_0 & 0 & l_2  & -l_3 &  & \\
-l_1 & -l_2 & 0 &  &  & 
\end{smallmatrix}
\right)$
&
a plane double line
&
8
&
\begin{tabular}{c}
strictly semistable,\\ but not polystable
\end{tabular}
\\ \hline
(f) 
&
$\left(
\begin{smallmatrix}
&  &  & 0 &l _0 &l_1 \\
 &  &  & -l_0 & 0 & l_2  \\
 &  &  & -l_1 & -l_2 & 0 \\
0 &l _0 &l_1 & &  & \\
-l_0 & 0 & l_2  &  &  & \\
-l_1 & -l_2 & 0 &  &  & 
\end{smallmatrix}
\right)$
&
\begin{tabular}{c}
 a line
together with its \\
full first order \\
infinitesimal \\
neighbourhood
\end{tabular}
&
6
&
polystable
\\ \hline
\end{tabular}
\medskip

\caption{Semistable matrices $M$ with vanishing Pfaffian over $\P^4$}
\label{tPfaffZero}
\end{table}

\begin{theorem}
Let $M$ be a semistable skew-symmetric $6 \times 6$ matrix with vanishing Pfaffian over $\P^4$. Then there exists linearly independent linear forms
$l_0,\dots,l_4$ and a matrix  $S \in \GL_6(\C)$ such that $SMS^t$ is one of the matrices in Table \ref{tPfaffZero}.
\end{theorem}

\begin{proof}
If $M$ is stable and $\Pf(M) = 0$ it must be of format $(d)$ or $(f)$ in Table \ref{tFormats}. If it is of format $(d)$ we obtain the normal form $(b)$ and $(c)$ of Table \ref{tPfaffZero} from Proposition \ref{pStableD}. If it is of format $(f)$ but not of format $(d)$ in Table \ref{tFormats} we obtain the normal form $(a)$ of Table \ref{tPfaffZero} from Proposition \ref{pStableF}.

\medskip

If $M$ is strictly semistable we can apply our Proposition \ref{pSemiStable}: 
The case $r=5$ cannot occur since we are in $\P^4$; $r=4$ gives case $(d)$, $r=3$ is case $(e)$ and $r=2$ is case $(f)$ in Table \ref{tPfaffZero}.
\end{proof}

\begin{proposition} \label{pOnlyOneNormalForm}
Let $M$ be a semi-stable skew-symmetric $6 \times 6$ matrix with vanishing Pfaffian over $\P^4$. Than $M$
is congruent to at most one of the normal forms in Table \ref{tPfaffZero},.
\end{proposition}

\begin{proof}
Let $Y$ be the rank $2$ locus of $M$ with the scheme structure given by the $4 \times 4$ Pfaffians and
$\Ytop$ the top dimensional component of $Y$. We compute $\Ytop$ for all matrices in Table \ref{tPfaffZero} using 
{\tt Macaulay2} see \cite{BB22M2}. The results are recorded in the third column of Table \ref{tPfaffZero}.
Since $\Ytop$ is different in each case, the claim follows.
\end{proof}

To determine the stability properties of the normal forms in Table \ref{tPfaffZero} we provide two methods to prove that a given matrix is {\sl not} of a particular format. The first method is by looking at the dimension of the space of quadrics spanned by the $4 \times 4$ Pfaffians $M$ which we denote by $d_4(M) := \dim \Pf_4(M)$.  Notice that this is invariant under the $\mathrm{GL}_6(\C)$-action.

\begin{table}
\begin{tabular}{|c|c|c|c|c|}
\hline
& Format & $d_4$ & generalized rows 
\\ \hline
$\mathrm{NS}_1$ 
&
$\left(\begin{smallmatrix}
0 & 0 & 0 & 0 & 0 &* \\
0 & 0 & * & * & * & *  \\
0 & * & 0 & * & * & * \\
0 & * & * & 0 & * & * \\
0 & * & * & *  & 0 & * \\
* & * & * & * & * & 0 
\end{smallmatrix}\right) $
&
$\le 11$
&
$\P^0 \subset Z_1$
\\ \hline
$\mathrm{NS}_2$
&
$\left(\begin{smallmatrix}
0 & 0 & 0 & 0 & * &* \\
0 & 0 & 0 & 0 & * & *  \\
0 & 0 & 0 & * & * & * \\
0 & 0 & * & 0 & * & * \\
* & * & * & *  & 0 & * \\
* & * & * & * & * & 0 
\end{smallmatrix}\right) $
&
$\le 10$
&
$\P^1 \subset Z_2$
\\  \hline
$\mathrm{NS}_3$
&
$\left(\begin{smallmatrix}
0 & 0 & 0 & * & * &* \\
0 & 0 & 0 & * & * & *  \\
0 & 0 & 0 & * & * & * \\
* & * & * & 0 & * & * \\
* & * & * & *  & 0 & * \\
* & * & * & * & * & 0 
\end{smallmatrix}\right)$
&
$\le 12$
&
$\P^2 \subset Z_3$
\\ \hline
$\mathrm{NSS}_1$
&
$\left(\begin{smallmatrix}
0 & 0 & 0 & 0 & 0 & 0 \\
0 & 0 & * & * & * & *  \\
0 & * & 0 & * & * & * \\
0 & * & * & 0 & * & * \\
0 & * & * & *  & 0 & * \\
0 & * & * & * & * & 0 
\end{smallmatrix}\right)$
&
$\le 5$
&
$\P^0 \subset Z_0$
\\ \hline
$\mathrm{NSS}_2$
&
$\left(\begin{smallmatrix}
0 & 0 & 0 & 0 & 0 &* \\
0 & 0 & 0 & 0 & 0 & *  \\
0 & 0 & 0 & * & * & * \\
0 & 0 & * & 0 & * & * \\
0 & 0 & * & *  & 0 & * \\
* & * & * & * & * & 0 
\end{smallmatrix}\right)$
& 
$\le 7$
&
$\P^1 \subset Z_1$
\\ \hline
$\mathrm{NSS}_3$
&
$\left(\begin{smallmatrix}
0 & 0 & 0 & 0 & * &* \\
0 & 0 & 0 & 0 & * & *  \\
0 & 0 & 0 & 0 & * & * \\
0 & 0 & 0 & 0 & * & * \\
* & * & * & *  & 0 & * \\
* & * & * & * & * & 0 
\end{smallmatrix}\right)$
& 
$\le 6$
&
$\P^3 \subset Z_2$
\\ \hline
\end{tabular}
\caption{Invariants of not stable and not semi-stable formats}
\label{tNSSinvariants}
\end{table}

\begin{proposition} \label{pStabilityViaD4}
For the formats $\mathrm{NS_i}$ and $\mathrm{NSS_i}$ the invariant $d_4$ is as listed in Table \ref{tNSSinvariants}.
In particular if $d_4(M) \ge 8$ then $M$ is semi-stable and  if $d_4(M) \ge 13$ then $M$ is stable.
\end{proposition}

\begin{proof}
We compute the invariant $d_4$ for the generic matrices of each format, i.e. those with linearly independent linear forms at the position of the $*$'s. For arbitrary entries the dimension of the space of quadrics spanned by the $4 \times 4$ Pfaffians can only be smaller than the dimension in the generic case. 

For example in the case $\mathrm{NSS_3}$ the $4 \times 4$ Pfaffians of $M$ are just the $2 \times 2$ minors of the $2 \times 4$ matrix at the position of the $*$'s. This gives at most $6$ independent $4 \times 4$ Pfaffians. In the other cases we compute $d_4$ in the generic case with {\tt Macaulay2} \cite{BB22M2}. 

If $d_4(M)$ is at least $8$ it cannot have one of the formats $\mathrm{NSS_i}$  since $d_4(M)$ is too large. Therefore it cannot be not semi-stable. Similarly if $d_4(M)$ is at least $13$ then $M$ cannot have formats $\mathrm{NS_i}$. Therefore in this case $M$ cannot be not stable.
\end{proof}

\begin{corollary} \label{cAtoEss}
The matrices of type $(a)$ to $(e)$ in Table \ref{tPfaffZero} are semi-stable.
\end{corollary}

\begin{proof}
We compute $d_4(M)$ for the matrices of type $(a)$ to $(f)$ in Table \ref{tPfaffZero}, see \cite{BB22M2}. The results are listed in the fourth
column of Table \ref{tPfaffZero}. We then apply Proposition \ref{pStabilityViaD4} in the cases $(a)$ to $(e)$.
\end{proof}

In some cases the invariant $d_4$ does not give enough information. For a finer invariant we introduce the
concept of generalized rows and their rank which we learned from David Eisenbud. 

\begin{definition}
Let $V,W$ be $\C$ vector spaces. After choosing bases of $V$ and $W$ we have a $1:1$ correspondence 
between $\dim W \times \dim W$ matrices of linear forms over $\P(V)$ and elements of
\[
	\Hom(V, W^* \otimes W^*).
\]
In this situation we call $\P(W)$ the space of {\sl generalized rows}. 
Considering the natural isomorphisms
\[
	\Hom(V, W^* \otimes W^*) \cong W^* \otimes W^* \otimes V^* \cong \Hom(W, V^*\otimes W^*)
\]
we see that (after choosing bases for $V$ and $W$) there is a $1:1$ correspondence between
$\dim W \times \dim W$ matrices of linear forms over $\P(V)$ and $\dim V \times \dim W$
matrices of linear forms over $\P(W)$. We denote this correspondence by
\[
	M \mapsto \widehat{M}.
\]
and call $\widehat{M}$ the {\sl flipped matrix} of $M$. 

We define the {\sl rank of a generalized row} $[w] \in \P(W)$ as the rank of $\widehat{M} (w)$. Indeed if
$[w]$ corresponds to a row of $M$, the rank of $[w]$ will be the dimension of the vector space spanned by the
linear forms in this row. We also let $Z_r \subset \P(W)$ be the space of generalized rows with rank at most $r$
with the scheme structure given by the $(r+1) \times (r+1)$ minors of $\widehat{M}$.

 For $S \in \GL(W)$ the operation $M \mapsto SMS^t$ corresponds to a change of basis in $W$. Therefore the
 schemes $Z_r(M)$ and $Z_r(SMS^t)$ in $\P(W)$ differ only by a linear transformation of $\P(W)$. 
 \end{definition} 

\begin{example}
Consider matrices of the form $\mathrm{NSS_3}$. The first four rows and all linear combinations of them have the first $4$ entries equal to 
zero. Therefore the space of  rank at most $2$ generalized rows $Z_2$ contains a $\P^3$. In the same way
the other formats of Table \ref{tNSSinvariants} have large linear spaces of low rank generalized rows. We have collected this information in Table \ref{tNSSinvariants}.
\end{example}

For a given matrix $M$ the flipped matrix can be easily computed:

\begin{proposition}
Let $V,W$ be $\C$-vector spaces with bases $l_0,\dots l_m$ of $V^*$ and $y_0,\dots,y_n$ of $W^*$. 
Let furthermore $M$ be a $\dim W \times \dim W$ matrix of linear forms over $\P(V)$ written in the above bases
and $\widehat{M}$ be the flipped matrix of $M$. Then
\[
 \widehat{M}_i = \frac{\partial}{\partial l_i} \bigl( (y_0, \dots , y_n) M  \bigr) , \quad i=0,\dots , m
\]
where $\widehat{M}_i$ is the $i$-th row of $\widehat{M}$. Notice that $\widehat{M}$ is a $\dim V \times \dim W$ matrix.
\end{proposition}

\begin{proof} 
It suffices to check for matrices $M$ with a single nonzero entry equal to one of the $l_i$, i.e. monomials in $W^* \otimes W^* \otimes V^*$. 
\end{proof}

\begin{example} \label{eFss}
Consider the matrix
\[
M= \left(\begin{smallmatrix}
&  &  & 0 &l _0 &l_1 \\
 &  &  & -l_0 & 0 & l_2  \\
 &  &  & -l_1 & -l_2 & 0 \\
0 &l _0 &l_1 & &  & \\
-l_0 & 0 & l_2  &  &  & \\
-l_1 & -l_2 & 0 &  &  & 
\end{smallmatrix}\right). 
\]
of type $(f)$. We compute the flipped matrix as
\[
	\widehat{M} = \left( \frac{\partial}{\partial l_i} \bigl( (y_0, \dots , y_5) M_f  \bigr) \right)
	= \left(\begin{smallmatrix}
{-{y}_{4}}&{y}_{3}&0&{-{y}_{1}}&{y}_{0}&0\\
      {-{y}_{5}}&0&{y}_{3}&{-{y}_{2}}&0&{y}_{0}\\
      0&{-{y}_{5}}&{y}_{4}&0&{-{y}_{2}}&{y}_{1}
\end{smallmatrix} \right)
\]
It can be checked by a computer algebra calculation \cite{BB22M2} that the rank $\le 1$ locus $Z_1(M)$ defined by the two by two minors of $\widehat{M}$ is empty. Therefore $M$ cannot be of format $\mathrm{NSS_1}$ or $\mathrm{NSS_2}$. The underlying reduced subscheme of $Z_2(M)$ can be computed (see \cite{BB22M2}) to be defined by the two by two minors of 
\[
\begin{pmatrix}
y_0 & y_1 & y_2\\
y_3 & y_4 & y_5
\end{pmatrix}.
\]
This is a Segre embedded $\P^1\times \P^2$ in $\P^5$, i.e. three-dimensional and irreducible and not linear. Therefore $M$ is also not of format $\mathrm{NSS_3}$ and thus $M$ is semi-stable. 
\end{example}

\begin{theorem}
Let $M$ be of type $(a)$ to $(f)$ in Table \ref{tPfaffZero}. Then $M$ has the stability property listed in the last column  of Table \ref{tPfaffZero}.
\end{theorem}

\begin{proof}
By Corollary \ref{cAtoEss} and Example \ref{eFss} all matrices in Table \ref{tPfaffZero} are semistable. If a matrix with vanishing Pfaffian is strictly semistable, then by Proposition \ref{pSemiStable} it must be of type $(d)$, $(e)$ or $(f)$ in Table \ref{tPfaffZero}. Now by Proposition \ref{pOnlyOneNormalForm} the types $(a)-(f)$ are distinct. Consequently types $(a)$, $(b)$ and $(c)$ are stable. 

\medskip 

Types $(d)$ and $(e)$ are not polystable, since type $(f)$ lies in the closure of their orbits.

\medskip

Finally matrices of type $(f)$ are polystable, since $d_4(M)$ is constant on orbits and it can only decrease on the closure. Since $d_4(M)$ is minimal for matrices of type $(f)$ this cannot happen. 
\end{proof}

\section{appendix}

Here we collect some known facts about formats.

\begin{table}[h!]
\begin{tabular}{|c|c|c|}
\hline
matrix & format & action\\
\hline
\begin{tabular}{c}$2 \times 2$ \\ symmetric \end{tabular}
& $\left(\begin{smallmatrix} * &  0\\  0 & 0\end{smallmatrix}\right) $
& $A\mapsto SAS^t$
\\ \hline
$2 \times 2$ 
& $
\left(\begin{smallmatrix} * &  0 \\  * & 0\end{smallmatrix}\right) \quad
\left(\begin{smallmatrix} * & * \\  0 & 0\end{smallmatrix}\right)
$
& $A\mapsto SAT$
 \\ \hline
$3 \times 3$ 
& $
\left(\begin{smallmatrix} * & * & 0 \\  * & * & 0\\ * & * & 0\end{smallmatrix}\right) \quad
\left(\begin{smallmatrix} * & * & * \\  * & * &  *\\ 0 & 0 & 0\end{smallmatrix}\right) \quad
\left(\begin{smallmatrix} * & * & * \\  * & 0 & 0 \\  *& 0 & 0\end{smallmatrix}\right) \quad
\left(\begin{smallmatrix} 0 & * & * \\  * & 0 & * \\ * &*  &0 \end{smallmatrix}\right) 
$
& $A\mapsto SAT$
\\ \hline
\begin{tabular}{c}$4 \times 4$ \\ skew \end{tabular}
& $
\left(\begin{smallmatrix} 0 & * & * & * \\  * & 0& 0 & 0 \\  * & 0& 0 & 0 \\  * & 0& 0 & 0 \end{smallmatrix}\right) \quad
\left(\begin{smallmatrix}  0 & * & * & 0 \\ * & 0 & * & 0 \\ * & * & 0 & 0   \\ 0 & 0 & 0 & 0   \end{smallmatrix}\right).
$
& $A\mapsto SAS^t$
\\ \hline
\end{tabular}
\caption{Formats of matrices of submaximal rank. The last format of the third row and all formats of the last row are skew.}
\label{tKnownFormats}
\end{table}

\begin{theorem}
Let $M$ be a square matrix of linear forms with vanishing determinant. If $M$ is as in the first column of Table \ref{tKnownFormats} then it has one of the formats depicted in the second column of the same row. 
\end{theorem}

\begin{proof}
{\bf 1st row:} The determinant of a symmetric $2 \times 2$ matrix defines a smooth conic in the $\P^2$ of all symmetric $2 \times 2$ matrices. The only linear spaces contained in a smooth conic are points. These correspond to the depicted format.

{\bf 2nd row:} The determinant of a $2 \times 2$ matrix defines a smooth quadric surface in the $\P^3$ of all $2 \times 2$ matrices. The only linear spaces contained in a smooth quadric surface are the lines of the two rulings. These correspond to the two formats.

{\bf 3rd row:} This is a theorem by Testa \cite{Te17}.

{\bf 4th row:} The Pfaffian of a skew-symmetric $4 \times 4$ matrix defines the Grassmannian  $G(2,4)$ in the space of all skew-symmetric $4 \times 4$ matrices. $G(2,4)$ parametrizes lines in a $\P^3$. The only linear spaces in $G(2,4)$ are 
\begin{itemize}
\item For each point in $P \in \P^3$ the $\P^2$ of lines containing this point
\item For each hyperplane $H \in \P^3$ the $\P^2$ of lines contained in $H$
\end{itemize}
These two types correspond to the two formats depicted in the last row.

\end{proof}

\providecommand{\bysame}{\leavevmode\hbox to3em{\hrulefill}\thinspace}
\providecommand{\href}[2]{#2}

\end{document}